\documentclass[11pt, leqno]{amsart}

\usepackage[T1]{fontenc}
\usepackage{mathtools}
\usepackage{amssymb}
\usepackage{microtype}
\usepackage[numbers]{natbib}
\usepackage[colorlinks=true,pagebackref=false,linktocpage=true]{hyperref}
\usepackage[english]{babel}
\usepackage{enumitem}

=\oddsidemargin \font\teneufm=eufm10 \font\seveneufm=eufm7
\font\fiveeufm=eufm5
\newfam\eufmfam
\textfont\eufmfam=\teneufm \scriptfont\eufmfam=\seveneufm
\scriptscriptfont\eufmfam=\fiveeufm

\newtheorem{theoremAlpha}{Theorem}

\newtheorem{stheorem}{Theorem}[section]

\newtheorem{example}[stheorem]{Example}
\newtheorem{prop}[stheorem]{Proposition}
\newtheorem{lem}[stheorem]{Lemma}
\newtheorem{coro}[stheorem]{Corollary}
\newtheorem{thm}[stheorem]{Theorem}

\newcommand{\Kbar}{\overline{K}}
\newcommand{\Br}{\mathrm{Br}}
\newcommand{\GG}{\mathbb{G}}
\newcommand{\sdim}{\mathrm{{}_S dim}}
\newcommand{\kdim}{\mathrm{{}_K dim}}

\defcitealias{SGA3}{SGA3}
\newcommand{\sga}[1]{\citepalias[#1]{SGA3}}

\makeatletter
\def\@seccntformat#1{%
  \protect\textup{\protect\@secnumfont
    \ifnum\pdfstrcmp{subsection}{#1}=0 \bfseries\fi
    \csname the#1\endcsname
    \protect\@secnumpunct
  }%
}  
\let\c@subsection\c@stheorem

\makeatother

\begin{document}

\title[Conjecture I for unirational groups]{Conjecture I for unirational algebraic groups over imperfect fields}

\author[A. Lourdeaux]{Alexandre Lourdeaux}
\address{Alexandre Lourdeaux - SUSTech International Center for Mathematics, Southern University of Science and Technology, Shenzhen 518055, China}

\author[A. Zidani]{Anis Zidani}
\address{Anis Zidani - Institut de Mathématiques de Jussieu-Paris Rive Gauche, Sorbonne Université, Campus Pierre et Marie Curie,
4, place Jussieu, 75252 Paris Cedex 05,
France}

\date{\today}

\begin{abstract} 
Serre's Conjecture I states that the first Galois cohomology set of any smooth connected linear algebraic group is trivial over a perfect field of cohomological dimension at most 1. We prove that this result remains valid for any unirational algebraic group, dropping the perfection assumption on the field. To do so, we rely on the theory of pseudo-reductive groups, combined with the structure of unirational wound unipotent groups and the recent theory of permawound unipotent groups. Finally, we extend several related results on Galois cohomology associated with the Conjecture.
\end{abstract}

\maketitle

\medskip


\medskip


\bigskip

\tableofcontents

\section{Introduction}

Let $K$ be a field of characteristic exponent $p$. We are interested in the vanishing of the first Galois cohomology sets $H^1(K,G)$ of linear algebraic $K$-groups $G$. In \cite[\S 2.4]{SerreConjI}, Serre formulated his Conjecture I: given a smooth connected linear algebraic $K$-group $G$, if $K$ is perfect and $\sdim(K) \leq 1$, then $H^1(K,G)$ is trivial. Here $\sdim(K) \leq 1$ means that the \emph{Serre dimension} of $K$ is at most $1$ (\cite[III.3.1]{CohoGalois}), {i.e.}, that the Brauer group of every finite extension of $K$ is trivial. The conjecture was proved by Steinberg in \cite[Thm.~I.9]{SteinbergProof}. In this paper, we establish a generalization of the conjecture.

\begin{theoremAlpha}[Thm.~\ref{theorem main result}]
\label{theorem h1 unirational}
Let $K$ be a field such that $\sdim(K) \leq 1$ and $[K:K^p] \leq p$, and let $G$ be a smooth connected linear algebraic group over $K$ which is $K$-unirational (i.e., there is a dominant {$K$-rational} map from an affine space to $G$). Then $H^1(K,G)=1$.
\end{theoremAlpha} 

Since every smooth connected linear algebraic $K$-group is $K$-unirational when $K$ is perfect, Theorem \ref{theorem h1 unirational} indeed recovers Serre's Conjecture I.
This extension is motivated by the principle that the structure of $K$-unirational algebraic $K$-groups is essentially driven by purely inseparable Weil restrictions of reductive groups, making them amenable to study since Weil restriction commutes with Galois cohomology.
The condition that $\sdim(K) \leq 1$ \textit{and} $[K:K^p] \leq p$ corresponds to $K$ having \emph{Kato dimension} at most $1$ (\cite[\S0, Def.]{Kato_galois}), denoted here by $\kdim(K) \leq 1$. 

There exist smooth connected unipotent groups whose first Galois cohomology sets are not trivial over imperfect fields $K$ such that $\kdim (K) \leq 1$.

\begin{example} \label{example unipotent h1 infini}
Let $k$ be an algebraically closed field of characteristic $p>0$, and consider the field $K=k(t)$, which has Kato dimension at most $1$. The $K$-subgroup \[ U=\{ (x,y) \in \mathbb{G}_{a,K}^2 \; | \; x+tx^{p^2} +y^{p^2}=0 \} \] is smooth, connected, and unipotent since it is isomorphic to $\mathbb{G}_a$ over $K^{1/p^2}$. We have an isomorphism
\[ K/\{ x+tx^{p^2} +y^{p^2} \; | \; (x,y) \in K^2 \} \overset{\sim}{\to} H^1(K,U) \]
obtained from the cohomological sequence induced by the short exact sequence 
\[ 1 \to U \to \mathbb{G}_a^2 \to \mathbb{G}_a \to 1 . \] 
Since the classes of $at^{-1}$, for $a \in k$, are all different, it follows that $H^1(K,U)$ is infinite.
\end{example} 

Classical examples of smooth connected unipotent groups with non trivial $H^1$ sets as in Example \ref{example unipotent h1 infini} are never $K$-unirational. For instance, the group from Example \ref{example unipotent h1 infini} cannot be $K$-unirational since it has only one $K$-point while $K$ is infinite (see \cite[1991--1992, 2]{TitsPseudoRed}). In fact, from these classical examples, we can even create \emph{pseudo-reductive} groups that are not $K$-unirational and whose $H^1$ sets are not trivial. Recall that a pseudo-reductive $K$-group is a smooth connected linear algebraic $K$-group that has no nontrivial smooth connected unipotent normal subgroup over $K$ (see \cite[Def.~1.1.1]{conrad_gabber_prasad_2015}).

\begin{example}
\label{contre exemple conjecture pseudo-reductif}
Let $K$ and $U$ be as in Example \ref{example unipotent h1 infini}. By \cite[Cor.~9.5]{Totaro}, there exists a commutative pseudo-reductive group $G$ over $K$ such that $G/T = U$ where $T$ is the maximal torus of $G$. The group $G$ cannot be $K$-unirational since $U$ is not. Since $\sdim(K) \leq 1$, by \cite[Lem.~A.2]{Cesnavicius}, we have $H^1(K,T)=0$ and $H^2(K,T)=0$. Therefore ${H^1(K,G) =H^1(K,U)}$, which is not trivial.
\end{example} 

We do not know whether Theorem \ref{theorem h1 unirational} holds if the assumption $[K:K^p]\leq p$ is dropped. The classical extension of Serre's Conjecture I to imperfect fields with $\sdim$ instead of $\kdim$ restricts the focus to reductive groups, as Serre himself suggested (\cite[\S 2.4, Rmk.~2)]{SerreConjI}). This was established by Borel and Springer \cite[\S 8.6]{GeneralisationBorelSpringer} by adapting Steinberg's approach.
\medskip

The proof of Theorem \ref{theorem h1 unirational} is distributed over Sections \ref{section from unirational groups to perfect groups}, \ref{section from the perfect to the pseudo-semisimple case}, \ref{section pseudo-semisimple case}, and finalized in \ref{section bijection h1}. Overall, it essentially amounts to a reduction to the reductive case.

\begin{itemize}
\item Section \ref{section from unirational groups to perfect groups}. We show that we can reduce the problem from $K$-unirational groups to perfect groups, that is, groups equal to their own derived subgroups. The reduction requires only the assumption $\sdim(K) \leq 1$, and the triviality of the first Galois cohomology set then holds for unirational solvable groups, see Proposition \ref{proposition h1 SolvableCase}.
\item Section \ref{section from the perfect to the pseudo-semisimple case}. We prove that, when $[K:K^p] \leq p$, the \emph{unipotent $K$-radical} (see \S \ref{paragraphe pseudored}) of a $K$-unirational algebraic $K$-group is $K$-unirational as well, see Theorem \ref{theorem unipotent radical unirational}. This implies that the problem reduces to the case of \emph{pseudo-semisimple groups}, that is, perfect pseudo-reductive groups. We use here the work of Rosengarten on permawound groups from \cite{Rosengarten_Permawound, rosengarten2024rigidityunirationalgroups}.  
\item Section \ref{section pseudo-semisimple case}. We define \emph{torsion primes} for pseudo-semisimple groups and use the structure theory from \cite{conrad_gabber_prasad_2015} to prove a pseudo-semisimple analogue of Serre's Conjecture I via a reduction to the semisimple case (Theorem \ref{theorem conjecture I pseudo semisimple}). Here, we assume only that, for torsion primes $l$, the $l$-primary components of the Brauer groups vanish, as in the classical refinement of the semisimple case \cite[Thm.~5.2.5]{GilleConjII}. We nevertheless assume that $[K:K^p]\leq p$ whenever $p\in\{2,3\}$ due to the intricate nature of pseudo-semisimple groups in these characteristics.
\end{itemize} 

A general smooth algebraic group $G$ over $K$ has a largest $K$-unirational subgroup $G^{\mathrm{uni}}$, and when $[K:K^p] \leq p$, the group $G/G^{\mathrm{uni}}$ has no nontrivial $K$-unirational subgroups (see \S \ref{subsection strongly wound quotient}). We have:

\begin{theoremAlpha}[Thm.~\ref{corollary bijective h1 strongly wound}]
\label{theorem bijection h1 strongly wound quotient}
Let $K$ be a field such that $\kdim(K) \leq 1$ and $G$ be a (not necessarily connected nor linear) smooth algebraic group over $K$. Then the quotient morphism ${G \to G/G^{\mathrm{uni}}}$ induces a bijection $H^1(K,G) \overset{\sim}{\rightarrow}H^1(K,G/G^{\mathrm{uni}})$.
\end{theoremAlpha} 

Theorem \ref{theorem bijection h1 strongly wound quotient} is a consequence of Theorem \ref{theorem h1 unirational} and of an improvement of \cite[III.2.4, Thm.~3]{CohoGalois} which is only stated for {perfect} fields. In Section \ref{section bijection h1}, we check that we can adapt the proof of \textit{loc.~cit.}~by using the notion of \emph{pseudo-Borel subgroups} from \cite{conrad_prasad_2016} instead of Borel subgroups.
\begin{theoremAlpha}[Thm.~\ref{theorem surjective h1}]
\label{theorem surjective h1 intro}
Let $K$ be a field such that $\sdim(K) \leq 1$ and let $G$ be a (not necessarily linear) smooth algebraic group over $K$. Then every $G$-homogeneous space $X$ is dominated by a $G$-torsor.
\end{theoremAlpha} 

Finally, over a field $K$ such that $\sdim(K) \leq 1$, the classical proof of Serre's Conjecture I and for its version for reductive groups (\cite{SteinbergProof, GeneralisationBorelSpringer}) shows at the same time that every reductive group over $K$ is {quasi-split}, meaning it admits a Borel subgroup over $K$, that is, a solvable parabolic subgroup. In general, a smooth connected linear group is said to be \textit{quasi-split} if it has a {pseudo-Borel} subgroup over $K$ (see \S \ref{pseudo-borels} for the definition). We get a generalization of this result as a consequence of Theorem \ref{theorem conjecture I pseudo semisimple} and using results on \emph{pseudo-inner forms} from \cite{conrad_prasad_2016}.

\begin{theoremAlpha}[Thm.~\ref{theorem pseudo-borel}]
\label{theorem quasisplit intro}
 Let $K$ be a field such that $\sdim(K) \leq 1$, assuming further that $\kdim(K)\leq1$ if $\mathrm{char}(K)\in\{2,3\}$. Then every smooth connected linear $K$-group is quasi-split.
\end{theoremAlpha}

\vspace{\baselineskip}

\noindent{\bf Acknowledgments.}${}$

The authors are deeply indebted to Philippe Gille for proposing this collaboration and for helpful discussions. They are also very grateful to K{\k e}stutis {\v C}esnavi{\v c}ius for his careful reading of the paper, and they warmly thank Zev Rosengarten for valuable discussions concerning Section \ref{section from the perfect to the pseudo-semisimple case}. Finally, they wish to thank Yong Hu who enabled the second author to visit the first author at SUSTech in Shenzhen. A version of Theorem \ref{theorem conjecture I pseudo semisimple} was also conjectured by Yong Hu.

The second author was supported by the project ``Group schemes, root systems, and related representations'' founded by the European Union - NextGenerationEU through Romania’s National Recovery and Resilience Plan (PNRR) call no.~PNRR-III-C9-2023-I8, Project CF159/31.07.2023, and coordinated by the Ministry of Research, Innovation
and Digitalization (MCID) of Romania.

\section{Notation, overview of used notions}
\label{section notation overview}

\subsection{Basic notation and conventions}
\begin{itemize}
    \item $K$ shall denote a field of characteristic exponent $p$. We fix an algebraic closure $\overline{K}$ of $K$, and call $K_s$ the separable closure of $K$ in $\overline{K}$. Set $\Gamma \coloneqq \mathrm{Gal}(K_s/K)$.
    \item By an algebraic $K$-group, we mean a group scheme over $K$ of finite type.
    \item For any finite field extension $L/K$, we write $\mathrm{R}_{L/K}(G)$ for the Weil restriction of an algebraic $L$-group $G$ (see \cite[App.~A.5]{conrad_gabber_prasad_2015}).
    \item $\mathcal{D}(G)$ denotes the derived subgroup of a smooth algebraic $K$-group $G$. We also write $G^\mathrm{ab}$ for $ G/\mathcal{D}(G)$.
    \item We adopt the convention that reductive groups are connected.
\end{itemize}

\subsection{Dimension of a field}
For any $l$ which is a prime or $1$\footnote{The inclusion of the case $l=1$ is a lightweight convention introduced solely to uniformize various propositions in the case $p=1$.}, write $\sdim_l(K) \leq 1$ if for all finite field extensions $L/K$, the $l$-torsion subgroup of the Brauer group $\Br(L)$ of $L$ is trivial.  Write $\sdim (K) \leq 1$ if $\sdim_l (K) \leq 1$ for all primes $l$. This is the Serre ($l$-)dimension of $K$. Write $\kdim (K) \leq 1$, meaning that $K$ has Kato dimension at most $1$, if $\sdim(K) \leq 1$ and $[K:K^p] \leq p$. It is clear that for any finite field extension $L/K$, if $\sdim_l(K) \leq 1$ for some prime $l$, then $\sdim_l(L) \leq 1$. Since $[K:K^p]\leq p$ implies $[L:L^p]\leq p$, we also have $\kdim(L) \leq 1$ for all finite extensions $L/K$ if $\kdim (K) \leq 1$.

In \cite[Def.~4.5.1]{GilleConjII}, Gille defines the notion of \emph{separable cohomological $l$-dimension}, but being of separable cohomological $l$-dimension at most $1$ and being of Serre $l$-dimension at most $1$ are equivalent by \cite[II.2.2, Cor.~2]{CohoGalois}. This enables us to cite the results of \cite{GilleConjII} with assumptions on Serre dimension being at most $1$.

\subsection{Unirational groups}\label{paragraphUnirat}
A finite type $K$-scheme $X$ is called \emph{($K$-)unirational} if there is a schematically dominant $K$-rational map $\mathbb{A}_K^n \dashrightarrow X$, that is, there exists a nonempty open subscheme $U \subset \mathbb{A}_K^n$ and a regular map $U \to X$ which does not factor through any proper closed subscheme of $X$. With this general definition of unirationality, if $G$ is a unirational algebraic $K$-group, then $G$ is smooth, connected, and linear, as explained at the beginning of the proof of \cite[Thm.~7.9]{rosengarten2024rigidityunirationalgroups}. This enables us to simply write ``unirational algebraic group'' instead of ``unirational smooth connected linear algebraic group''. By \cite[Thm.~1.6]{rosengarten2024rigidityunirationalgroups}, $K$-unirationality for algebraic $K$-groups can be checked over a separable extension.

\subsection{Split and wound unipotent groups}
A smooth connected unipotent $K$-group $U$ is called \emph{split} if it admits a composition series whose successive quotients are isomorphic to $\mathbb{G}_{a,K}$. It is called \emph{($K$-)wound} (see \cite[Def.~B.2.1]{conrad_gabber_prasad_2015}) if every scheme map from the affine line $\mathbb{A}_K^1$ into $U$ is constant to some $K$-point of $U$. By \cite[Prop.~B.3.2]{conrad_gabber_prasad_2015}, a smooth connected unipotent $K$-group $U$ is wound if, and only if, it has no central subgroup isomorphic to the additive group $\mathbb{G}_a$. In general, there exists a unique smooth connected normal split unipotent $K$-subgroup $U_\mathrm{split}$ of $U$ such that $U/U_\mathrm{split}$ is wound (this is \cite[Thm.~B.3.4]{conrad_gabber_prasad_2015}).

\subsection{Quasi-reductive, pseudo-reductive, and pseudo-semisimple groups}
\label{paragraphe pseudored}
A \emph{pseudo-reductive $K$-group} (resp.~\emph{quasi-reductive $K$-group}) is a smooth connected linear algebraic group $G$ over $K$ which does not admit any nontrivial smooth connected normal unipotent (resp.~$K$-split) $K$-subgroup. A \emph{pseudo-semisimple} $K$-group is a perfect pseudo-reductive {$K$-group}. In general, for a smooth connected linear algebraic $K$-group $G$, we define its \emph{unipotent $K$-radical} ${R}_{u,K}(G)$ (resp.~\emph{split unipotent $K$-radical} ${R}_{us,K}(G)$) as its largest smooth connected normal (resp.~$K$-split) unipotent $K$-subgroup. The quotient $G/{R}_{u,K}(G)$ (resp.~$G/{R}_{us,K}(G)$) is always pseudo-reductive (resp.~quasi-reductive) and there is a natural extension \[1 \to {R}_{u,K}(G) \to G \to G/{R}_{u,K}(G) \to 1 ,\] so that the study of linear algebraic groups splits roughly into the study of unipotent groups and of pseudo-reductive groups. Note that ${R}_{u,K}(G)_{\mathrm{split}}={R}_{us,K}(G)$ by \cite[Cor.~B.3.5]{conrad_gabber_prasad_2015}.

\subsection{The largest smooth subgroup of an algebraic $K$-group}
\label{largestsmsubgrp}
By \cite[Lem.~C.4.1 and Rmk.~C.4.2]{conrad_gabber_prasad_2015}), any algebraic $K$-group $G$ admits a largest smooth subgroup $G^\mathrm{sm}$. It satisfies $G^\mathrm{sm}(K')=G(K')$ for all separable extensions $K'/K$. It is functorial in $G$ and commutes with separable base change.

\subsection{The subgroup \boldmath{$\mathcal{D}^\infty(G)$}}
If $G$ is a smooth connected algebraic $K$-group, there exists an integer $n \geqslant 0$ such that $\mathcal{D}^n(G)$ is perfect for dimension reasons. In other words, the sequence $(\mathcal{D}^k(G))_{k \geqslant 0}$ stabilizes at some point. We denote by $\mathcal{D}^\infty(G)$ the group $\mathcal{D}^n(G)$.

\subsection{Cohomology}
Given any algebraic $K$-group $G$, $H^i(K,G)$ stand for the \textit{fppf} $i$-th cohomology set/group of $G$. If $G$ is smooth, $H^i(K,G)$ is identified with the étale, or Galois cohomology set/group.
Recall that, given a finite field extension $K'/K$ and an algebraic $K'$-groups $G'$, the canonical map $H^1(K,\mathrm{R}_{K'/K}(G')) \to H^1(K',G')$ is injective by \sga{XXIV, Prop.~8.2}. It is moreover bijective when $G'$ is smooth by \sga{XXIV, Rmk.~8.5}.

\section{From unirational groups to perfect groups}
\label{section from unirational groups to perfect groups}

The goal of this section is to reduce the proof of the triviality of $H^1(K,G)$ for unirational algebraic $K$-groups $G$ to the case where $G$ is perfect, assuming that $\sdim(K) \leq 1$. In particular, $H^1(K,G)$ is trivial for unirational solvable groups $G$. We achieve this reduction by first treating commutative groups.

\begin{prop}
\label{proposition h1 commutative unipotent}
Let $U$ be a unirational commutative unipotent algebraic group over $K$. \linebreak If $\sdim_p (K) \leq 1$, then we have $H^1(K,U)=1$.
\end{prop}

\begin{proof}
The map $U\rightarrow U/U_{\mathrm{split}}$ induces an exact sequence in cohomology
\[ {H^1(K,U_{\mathrm{split}})} \to {H^1(K,U)} \to {H^1(K,U/U_{\mathrm{split}})} .\] 
Since $H^1(K,U_{\mathrm{split}})$ is trivial, we just need to prove that $H^1(K,U/U_{\mathrm{split}})$ is trivial, hence we may assume that $U$ is $K$-wound. The group $U$ being $K$-wound and unirational, it is by \cite[Prop.~2.5]{achet:hal-02358528} a quotient of a group $W$ which is a product of groups of the form 
\begin{equation}
\label{equation prop h1 commutative unipotent}
    \mathrm{R}_{L/K}\left( \mathrm{R}_{K'/L}(\GG_{m,K'})/\GG_{m,L} \right)
\end{equation} 
for a finite field extension $K'/K$ and where $L$ is the separable closure of $K$ in $K'$. Let us denote by $V$ the kernel of $W\to U$. Since $W$ is unipotent, so is $V$. We have an exact sequence 
\[ {H^1(K,W)} \to {H^1(K,U)} \to {H^2(K,V)} \]
where $H^2(K,V)$ is trivial by \cite[Lem.~3.3]{TossiciVistoli}. To prove that $H^1(K,W)$ is trivial, it suffices to consider a group $W_0$ of the form \eqref{equation prop h1 commutative unipotent}. On the one hand, we have $$H^1(K,W_0)=H^1(L, \mathrm{R}_{K'/L}(\GG_{m})/\GG_{m}), $$ and on the other hand, we have $H^1(L, \mathrm{R}_{K'/L}(\GG_{m}))=H^1(K',\GG_m)= 0$ by Hilbert's Theorem 90. Moreover, the long exact sequence associated with the quotient $\mathrm{R}_{K'/L}(\GG_{m})/\GG_{m}$ is 
\[\underset{0}{\underbrace{{H^1(L,\mathrm{R}_{K'/L}(\GG_{m}))}}} \to {H^1(L,\mathrm{R}_{K'/L}(\GG_{m})/\GG_{m})} \to {H^2(L,\GG_m)},\]
and the $p$-primary component of $H^2(L,\GG_m)=\mathrm{Br}(L)$ is trivial since $\sdim_p(K) \leq 1$, while $H^1(L,\mathrm{R}_{K'/L}(\GG_{m})/\GG_{m})$ is $p$-torsion as $\mathrm{R}_{K'/L}(\GG_{m})/\GG_{m}$ is unipotent. Therefore, $H^1(L,\mathrm{R}_{K'/L}(\GG_{m})/\GG_{m})$ is trivial, and consequently, so is $H^1(K,W_0)$. We thus find that $H^1(K,W)$, and hence $H^1(K,U)$, are trivial.
\end{proof}

\begin{prop}
\label{proposition h1 commutative uni}
Let $C$ be a unirational commutative algebraic group over $K$. \linebreak If ${\sdim (K) \leq 1}$, then we have $H^1(K,C)=0$.
\end{prop}

\begin{proof}
Let $T$ be the maximal torus of $C$. Set $U\coloneqq C/T$. By dévissage, $H^1(K,C)$ is trivial if both $H^1(K,T)$ and $H^1(K,U)$ are. By \cite[Prop.~4.6.2]{GilleConjII}, $H^1(K,T)$ is trivial. Regarding $H^1(K,U)$, it is trivial due to Proposition \ref{proposition h1 commutative unipotent} since it is unirational as a quotient of $C$.
\end{proof}

\begin{lem}\label{DGUnirat}
    Let $G$ be a unirational algebraic group over $K$. Then $\mathcal{D}(G)$ is unirational.
\end{lem}

\begin{proof}
    By \sga{VIB, Prop.~7.4}, there is a surjective scheme map $(G \times_K G)^N \to \mathcal{D}(G)$, which is a product of the commutator map $G \times_K G \to \mathcal{D}(G)$. Hence the result since $G$ is unirational.
\end{proof}

\begin{prop}
\label{proposition h1 SolvableCase}
Let $G$ be a unirational algebraic group over $K$. If $\sdim (K) \leq 1$, then the inclusion $\mathcal{D}^\infty(G)\subset G$ induces a surjection
 $$H^1(K,\mathcal{D}^\infty(G)) \rightarrow H^1(K,G).$$
In particular, $H^1(K,G)$ is trivial if $G$ is further assumed to be solvable.
\end{prop}

\begin{proof}
    The exact sequence induced by $G\rightarrow G^\mathrm{ab}$ induces an exact sequence in cohomology: \[ {H^1(K,\mathcal{D}(G))} \to  {H^1(K,G)} \to {H^1(K,G^\mathrm{ab})} .\]
    The group $G^\mathrm{ab}$ is unirational as $G$ is, so by Proposition \ref{proposition h1 commutative uni} the group $H^1(K,G^\mathrm{ab})$ is trivial, thus $H^1(K,\mathcal{D}(G)) \rightarrow H^1(K,G)$ is surjective. Since $\mathcal{D}(G)$ is unirational by Lemma \ref{DGUnirat}, it then suffices to iterate the process to get the result.
\end{proof}

\begin{coro}
Let $G$ be a unirational $K$-group and $C$ be a Cartan subgroup of $G$. \linebreak If $\sdim (K) \leq 1$, then $H^1(K,C)$ is trivial.
\end{coro}

\begin{proof}
     Since $G$ is unirational, so is $C$ by \cite[Prop.~7.12]{rosengarten2024rigidityunirationalgroups}. Moreover, $C$ is nilpotent (see \sga{XII, Thm.~6.6.c)}), hence solvable. It now remains to apply Proposition \ref{proposition h1 SolvableCase} to $C$. 
\end{proof}

\section{From the perfect to the pseudo-semisimple case}
\label{section from the perfect to the pseudo-semisimple case}

The goal of this section is to prove that for any unirational smooth connected linear algebraic $K$-group $G$, the group ${R}_{u,K}(G)$ is unirational when $[K:K^p]\leq p$. This is done by first reducing to the case of a quasi-reductive commutative algebraic $K$-group, and then using the theory of permawound unipotent $K$-groups developed by Rosengarten in \cite{Rosengarten_Permawound} to deal with this case. Consequently, assuming moreover that $G$ is perfect and that $\kdim(K)\leq 1$, the triviality of $H^1(K,G)$ reduces to the case where $G$ is pseudo-semisimple. Indeed, the unirationality of ${R}_{u,K}(G)$ implies that $H^1(K,{R}_{u,K}(G))=1$ by Proposition \ref{proposition h1 SolvableCase}.

\begin{prop}
\label{proposition extension unipotent par tore unirational}
Let $G$ be a unirational commutative quasi-reductive algebraic $K$-group. Assume that $[K:K^p] \leq p$. Then ${R}_{u,K}(G)$ is unirational.
\end{prop}

\begin{proof}
Let $T$ be the maximal torus of $G$. Set $U \coloneqq G/T$ and  $p\coloneqq G\rightarrow U$. By \sga{XVII, Thm.~6.1.1.A)ii)}, the extension $p^{-1}(U_{\mathrm{split}})$ has a group section $s\colon U_{\mathrm{split}}\rightarrow G$, so that $s(U_{\mathrm{split}})\subset {R}_{us,K}(G)=1$. Hence $U$ is $K$-wound. Note that $G/R_{u,K}(G)$ is a pseudo-reductive extension of $U/p(R_{u,K}(G))$ by $T$. So, the same reasoning shows that $U/p(R_{u,K}(G))$ is {$K$-wound}. Also, $U$ is unirational as a quotient of $G$. Hence, $U$ is permawound by \cite[Prop.~9.6]{rosengarten2024rigidityunirationalgroups}. By \cite[Prop.~5.5 and Prop.~6.9]{Rosengarten_Permawound}, the group $R_{u,K}(G)\cong p(R_{u,K}(G))$ is permawound, hence unirational by \cite[Prop.~9.7]{rosengarten2024rigidityunirationalgroups}.
\end{proof}

\begin{prop}
\label{ReductionCasQREdComm}
Let $G$ be a unirational algebraic $K$-group. Set $G'\coloneqq G/R_{us,K}(G)$ 
and take $T'$, a maximal torus of $G'$. Then ${Z_{G'}(T')}^\mathrm{ab}$ is a unirational commutative quasi-reductive algebraic $K$-group, and we have the exact sequence:
\[1 \to U \to {R}_{u,K}(G) \to {R}_{u,K}({Z_{G'}(T')}^\mathrm{ab}) \to 1,\]
where $U$ is an unirational unipotent algebraic $K$-group.
\end{prop}

\begin{proof}
Set $Z'\coloneqq Z_{G'}(T')$. Note that ${R}_{u,K}(G') \subset Z'$ since the action of $T'$ on ${R}_{u,K}(G')$ is trivial by \cite[Prop.~B.4.4]{conrad_gabber_prasad_2015}. Also, by \sga{XII, Thm.~7.1.e)}, we have an isomorphism $Z{'/{R}_{u,K}(G')\overset{\sim}{\to} Z_{G'/{R}_{u,K}(G')}(T')}$ where the latter group is pseudo-reductive by \cite[Prop.~1.2.4]{conrad_gabber_prasad_2015}. Hence ${R}_{u,K}(G') = {R}_{u,K}(Z')$.
Furthermore, since $G'$ is unirational, so is $Z'$ by \cite[Prop.~7.12]{rosengarten2024rigidityunirationalgroups}, and hence the same holds for $\mathcal{D}(Z')$ by Lemma \ref{DGUnirat}. Now, since $Z'$ is nilpotent, we have an isomorphism $Z'_{\overline{K}}\cong T'_{\overline{K}} \times (Z'/T')_{\overline{K}}$, so that $\mathcal{D}(Z')$ is unipotent. Let us then consider the map
$${R}_{u,K}(G)\twoheadrightarrow {R}_{u,K}(G)/{R}_{us,K}(G)\overset{\sim}{\to} {R}_{u,K}(G')={R}_{u,K}(Z')\twoheadrightarrow {R}_{u,K}(Z')/\mathcal{D}(Z') \overset{\sim}{\to}{R}_{u,K}({Z'}^\mathrm{ab}).$$
Its kernel $U$ is an extension of $\mathcal{D}(Z')$ by ${R}_{us,K}(G)$. Since ${R}_{us,K}(G)$ is permawound by \cite[Prop.~5.3]{Rosengarten_Permawound}, the group $U$ is unirational by \cite[Prop.~5.3]{Rosengarten_extuni}.
\end{proof}

\begin{thm}
\label{theorem unipotent radical unirational}
Let $G$ be a unirational algebraic $K$-group. Assume that $[K:K^p] \leq p$. Then ${R}_{u,K}(G)$ is unirational.
\end{thm}

\begin{proof}
We keep the notation of Proposition \ref{ReductionCasQREdComm}. By Proposition \ref{proposition extension unipotent par tore unirational}, ${R}_{u,K}(Z_{G'}(T')^\mathrm{ab})$ is unirational. Proposition \ref{ReductionCasQREdComm} implies that ${R}_{u,K}(G)$ is an extension of unirational unipotent algebraic $K$-groups. It is then unirational by \cite[Thm.~2.4]{Rosengarten_extuni}.
\end{proof}

\section{The pseudo-semisimple case}
\label{section pseudo-semisimple case}

To study the first Galois cohomology sets of pseudo-semisimple groups, we use the structure theory one can find in \cite{conrad_gabber_prasad_2015}. We want to express a Conjecture I for pseudo-semisimple groups using the assumption $\sdim_l(K) \leq 1$ only for some primes $l$, depending on the group under consideration.

\subsection{The torsion primes of a pseudo-semisimple group}
Fix a pseudo-semisimple algebraic \mbox{$K$-group} $G$ and consider $\overline{G}\coloneqq G_{\overline{K}}/ {R}_{u,\overline{K}}(G_{\overline{K}})$. Define the \emph{set of torsion primes} $S(G)$ of $G$ as $S(\overline{G})$, which is the usual set of torsion primes of the semisimple group $\overline{G}$ as defined in \cite[\S 4.8 and \S 5.1]{GilleConjII}. It only depends on the type of $\overline{G}$. Also, for any finite extension $K/k$, we have $S(\mathrm{R}_{K/k}(G))=S(G)$. Indeed, if $K/k$ is separable, $\mathrm{R}_{K/k}(G)_{\Kbar}$ is a product of copies of $G_{\Kbar}$. If $K/k$ is purely inseparable, the adjunction $\mathrm{R}_{K/k}(G)_{K}\rightarrow G$ has a smooth, connected, unipotent kernel (\cite[Prop. A.5.11]{conrad_gabber_prasad_2015}), yielding isomorphic reductive quotients over $\Kbar$.

\begin{lem}
\label{lemme h2 tame groups}
Let $q\colon H\rightarrow G$ be a central quotient of smooth connected perfect linear algebraic $K$-groups. Set ${Z\coloneqq\ker(q)}$. Then $S(H/R_{u,K}(H)) = S(G/R_{u,K}(G))$. Moreover, if $\sdim_l(K) \leq 1$ for all $l\in S(G/R_{u,K}(G))$, then $H^2(K,Z)=0$.
\end{lem}

\begin{proof}
Denote by $\mu$ the largest multiplicative subgroup of the commutative group $Z$. By dévissage, it suffices to show that $H^2(K,Z/\mu)$ and $H^2(K,\mu)$ are trivial. Since $Z/\mu$ is unipotent, the first case comes from \cite[Lem.~3.3]{TossiciVistoli}. Next, observe that the map $q$ induces over $\overline{K}$ a surjective map $\overline{q} \colon \overline{H} \to \overline{G}$ between the semisimple quotients. Since ${R_{u,\overline{K}}(H_{\overline{K}})\rightarrow {R}_{u,\overline{K}}(G_{\overline{K}})}$ is onto (see \cite[Cor.~14.11]{BorelLinAlgGrp}), the Snake Lemma implies that $\overline{q}$ is a central isogeny with kernel $\mu_{\overline{K}}$. This gives ${S(H/R_{u,K}(H))=S(\overline{H})=S(\overline{G})=S(G/R_{u,K}(G))}$. Consider $\overline{G}^\mathrm{sc}\rightarrow \overline{G}$, the simply connected cover of $\overline{G}$. By \cite[2.2.2]{BourbakiSerre}, $Z(\overline{G}^\mathrm{sc})$ is $n$-torsion for an integer $n$ whose prime divisors are in $S(\overline{G})$. The same holds for $\mu_{\overline{K}}$ since it is a quotient of $Z(\overline{G}^\mathrm{sc})$. Therefore, if $\sdim_l(K) \leq 1$ for all $l\in S(G/R_{u,K}(G))$, then $H^2(K,\mu)=0$ by \cite[Prop.~4.6.2]{GilleConjII}.
\end{proof}

\subsection{Classification of pseudo-semisimple groups} 
\label{ClassificationPseudoSS}
Assume that $[K:K^2]\leq 2$ if $p=2$. Let $G$ be a pseudo-semisimple $K$-group. Then, by \cite[Thm.~10.2.1]{conrad_gabber_prasad_2015}, $G$ is a product $G_1\times G_2$ where $G_1$ is \emph{generalized standard} (in the sense of \cite[Def.~10.1.9]{conrad_gabber_prasad_2015}, not \cite[Def. 9.1.7]{conrad_prasad_2016}), and $G_2$ is \emph{totally nonreduced} (see \cite[Def.~10.1.1]{conrad_gabber_prasad_2015}) or trivial. Note that $G_2$ is trivial when $p>2$ (see \cite[Thm.~2.3.10]{conrad_gabber_prasad_2015}). Although we bypass their explicit definitions, both types of groups are well classified. First, \cite[Rmk.~10.1.11]{conrad_gabber_prasad_2015} asserts that $G_1$ is a central quotient of a perfect group $\mathrm{R}_{K'/K}(G')$ where $K'$ is a finite reduced $K$-algebra ({i.e.}, a product of finitely many finite field extensions of $K$) and $G'$ is a $K'$-group whose fibers over the component fields of $K'$ are either \emph{absolutely simple and semisimple simply connected} or, if $p \in \{2, 3\}$, \emph{basic exotic} (\cite[Def.~7.2.6]{conrad_gabber_prasad_2015}). Second, $G_2$ is either trivial or, by \cite[Prop.~10.1.4]{conrad_gabber_prasad_2015}, of the form $\mathrm{R}_{K''/K}(G'')$ for some finite reduced $K$-algebra $K''$ and a $K''$-group $G''$ that is \emph{basic nonreduced} (\cite[Def.~10.1.2]{conrad_gabber_prasad_2015}). Ultimately, $G$ is a central quotient of a pseudo-semisimple group $\mathrm{R}_{L/K}(H)$, where $L/K$ is a finite reduced $K$-algebra and where $H$ is a $L$-group such that each of its fibers is one of the three classes of \emph{basic groups} cited above. The group $G$ is said to be \emph{standard} if the only basic group appearing in this description is the simply connected semisimple one. In particular, $G$ is always standard when $p>3$.

\begin{thm}
\label{theorem conjecture I pseudo semisimple}
Let $G$ be a pseudo-semisimple $K$-group. Suppose that $\sdim_l(K)\leq 1$ for all ${l\in S(G)}$. If $G$ is not standard, further assume that $[K:K^p]\leq p$. Then $H^1(K,G)$ is trivial.
\end{thm}

\begin{proof}
We return to the notation of \S\ref{ClassificationPseudoSS}. Denote by $Z$ the kernel of the central quotient $\mathrm{R}_{L/K}(H)\rightarrow G$. From Lemma \ref{lemme h2 tame groups}, $H^2(K,Z)$ is trivial, so by dévissage it remains to prove that $H^1(K,\mathrm{R}_{L/K}(H))$ is trivial. Observe that $H^1(K,\mathrm{R}_{L/K}(H))=H^1(L,H)$. The latter $H^1$ set decomposes into the product of the $H^1$ sets of the fibers of $H$ over the component fields of $L$, so that we are reduced to the case where $H$ is a basic group. For the semisimple factors, this is \cite[Thm.~5.2.5]{GilleConjII}. For the basic exotic factors, this is \cite[Prop.~7.3.3.(1)]{conrad_gabber_prasad_2015} (here we need $[K:K^p] =p$, otherwise \textit{loc.~cit.}~does not cover all possible cases). For the basic nonreduced factors, this is \cite[Prop.~9.9.4.(1)]{conrad_gabber_prasad_2015}.
\end{proof}

\section{Quasi-splitness results}
\label{section triviality h1}

The classical proof of Serre's Conjecture I by Steinberg proves simultaneously that every reductive group $G$ over $K$ is quasi-split if $\sdim(K) \leq 1$. This means that $G$ admits a Borel subgroup over $K$, that is, a solvable parabolic subgroup, or equivalently, a minimal parabolic $K_s$-subgroup defined over $K$. We deduce from Theorem \ref{theorem conjecture I pseudo semisimple} a similar result with pseudo-Borel subgroups. This is Theorem \ref{theorem pseudo-borel} below. 
\subsection{Pseudo-parabolic subgroups}
\label{pseudo-parabolics}
Let $G$ be a smooth connected linear algebraic group over $K$. For any cocharacter $\lambda \colon \mathbb{G}_m \to G$, let $P_G(\lambda)$ be the subgroup of $G$ whose $R$-points are, for all $K$-algebras $R$, the $g \in G(R)$ such that the scheme morphism $\mathbb{G}_{m,R} \to G_R$, $t \mapsto \lambda(t)g\lambda(t)^{-1}$ can be extended to $\mathbb{A}_{R}^1$. The groups $P_G(\lambda)$ are smooth and connected by \cite[Prop.~2.1.8]{conrad_gabber_prasad_2015}. We then define a \emph{pseudo-parabolic subgroup} of $G$ to be a subgroup of the form ${R}_{u,K}(G) P_G(\lambda)$. When $G$ is reductive, the pseudo-parabolic subgroups are exactly the parabolic subgroups (see \cite[Prop.~2.2.9]{conrad_gabber_prasad_2015}). 

\subsection{Pseudo-borel subgroups}
\label{pseudo-borels}
In parallel to the reductive case, define a \emph{pseudo-Borel $K$-subgroup} $B$ of $G$ as a subgroup such that $B_{K_s}$ is minimal among the pseudo-parabolic $K_s$-subgroups of $G_{K_s}$ (see \cite[\S C.2]{conrad_prasad_2016}). Equivalently, by \cite[Prop.~3.5.4]{conrad_gabber_prasad_2015}, a pseudo-Borel $K$-subgroup of $G$ is a solvable pseudo-parabolic subgroup. The group $G$ is said to be \textit{quasi-split} if it has a pseudo-borel subgroup over $K$. 

\begin{thm}
\label{theorem pseudo-borel}
Let $G$ be a smooth connected linear algebraic $K$-group. Suppose that ${\sdim(K)\leq 1}$. If $G/{R}_{u,K}(G)$ is not standard, further assume that $[K:K^p]\leq p$. Then $G$ is quasi-split.
\end{thm}

To prove the theorem, let us introduce the notion of pseudo-inner forms from \cite{conrad_prasad_2016}, which generalizes the notion of inner forms for reductive groups to pseudo-reductive groups. 

\subsection{Pseudo-inner forms}
\label{PseudoInnerForms} Let $G$ be a pseudo-reductive $K$-group and $C$ be a Cartan subgroup of $\mathcal{D}(G)$. Consider the group functor $\mathrm{Aut}_{\mathcal{D}(G)/K}$ (resp.~$\mathrm{Aut}_{\mathcal{D}(G),C}$) of automorphisms of the pseudo-semisimple $K$-group $\mathcal{D}(G)$ (resp.~that fixes $C$ pointwise). It is representable by a linear algebraic $K$-group by \cite[Prop.~6.2.2]{conrad_prasad_2016} (resp.~by \cite[Thm.~2.4.1]{conrad_gabber_prasad_2015}).
Set $Z_{\mathcal{D}(G),C}\coloneqq\mathrm{Aut}_{\mathcal{D}(G),C}^\mathrm{sm}$. From \cite[Prop.~6.2.4]{conrad_prasad_2016} there is an isomorphism 
$$(\mathcal{D}(G) \rtimes Z_{\mathcal{D}(G),C})/C\overset{\sim}{\to}(\mathrm{Aut}^{\mathrm{sm}}_{\mathcal{D}(G)/K})^0,$$ 
where the semi-direct product is given by the natural action, and $C$ is embedded in it via the map $c \mapsto (c^{-1},\mathrm{int}(c))$. Furthermore, the action of $(\mathrm{Aut}_{\mathcal{D}(G)/K}^\mathrm{sm})^0$ on $\mathcal{D}(G)$ extends to $G$ (see \cite[Lem.~C.2.3]{conrad_prasad_2016}), and as \cite[Def.~C.2.4]{conrad_prasad_2016}, define a \emph{pseudo-inner form} of $G$ as a group obtained by twisting $G$ via a cocycle in the image of \[ H^1(K,(\mathrm{Aut}^{\mathrm{sm}}_{\mathcal{D}(G)/K})^0) \to H^1(K,\mathrm{Aut}_{G,K}), \] where $\mathrm{Aut}_{G,K}$ is the étale sheaf of automorphisms of $G$. 

\begin{prop}
\label{proposition quasi split pseudo inner form}
    Let $G$ be a smooth connected linear algebraic $K$-group. Suppose that ${\sdim(K)\leq 1}$. Then $G$ admits a quasi-split pseudo-inner form.
\end{prop}

\begin{proof}
We can reduce the question to pseudo-reductive groups by \cite[Prop.~2.2.10]{conrad_gabber_prasad_2015}. If $G$ is pseudo-reductive, by \cite[\S C.2.9]{conrad_prasad_2016} we know that the existence of a quasi-split pseudo-inner form is equivalent to the triviality of the $H^2$ group of a certain pseudo-reductive commutative $K$-group. This $H^2$ group is trivial by \cite[Cor.~A.4]{Cesnavicius}.
\end{proof}

\begin{proof}[Proof of Theorem \ref{theorem pseudo-borel}]
By \cite[Prop.~2.2.10]{conrad_gabber_prasad_2015}, we can suppose that $G$ is pseudo-reductive. Since ${\sdim(K)\leq 1}$, $G$ admits a quasi-split pseudo-inner form $G^q$ by Proposition \ref{proposition quasi split pseudo inner form}, so it suffices to prove that $H^1(K,(\mathrm{Aut}_{\mathcal{D}(G^q)/K}^\mathrm{sm})^0) =1$. By dévissage, it holds as soon as $H^1(K,\mathcal{D}(G^q))$, $H^1(K,Z_{\mathcal{D}(G^q),C})$, and $H^2(K,C)$ are trivial, where $C$ is any Cartan subgroup of $\mathcal{D}(G^q)$. First, \cite[Proof of Prop.~C.2.8]{conrad_prasad_2016} yields $H^1(K,Z_{\mathcal{D}(G^q),C})=1$ since $\mathcal{D}(G^q)$ is quasi-split. Moreover, $H^2(K,C)=1$ by \cite[Cor.~A.4]{Cesnavicius}. Finally, $H^1(K,\mathcal{D}(G^q))=1$ thanks to Theorem \ref{theorem conjecture I pseudo semisimple}.
\end{proof}

Here is a direct consequence for Bruhat--Tits theory.

\begin{prop}
    \label{CoroBT}
    Let $G$ be a reductive group over a henselian discretely valued field $K$, quasi-split over the maximal unramified extension $K^\mathrm{unr}$. Let $\kappa$ be the residue field of $K$, and $q$ be its characteristic. Suppose that $\sdim(\kappa) \leq 1$ and that $[\kappa:\kappa^q] \leq q$ if $q\in \{2,3\}$. Then $G$ is residually quasi-split (i.e., the Bruhat--Tits building of $G_{K^\mathrm{unr}}$ admits a \mbox{${\mathrm{Gal}(K^\mathrm{unr}/K)}$-invariant} chamber).
\end{prop}

\begin{proof}
    We have by \cite[Prop.~1.6.2]{macau_mods_00007771} equivalent conditions for $G$ being residually quasi-split. Condition (5) is actually satisfied thanks to Theorem \ref{theorem pseudo-borel}.
\end{proof}

\section{Main results}
\label{section bijection h1}

Suppose that $\kdim(K) \leq 1$. This final section contains the main results of the paper. We establish a variant of Serre's Conjecture I for unirational algebraic $K$-groups. Moreover, for any smooth algebraic $K$-group $G$, we introduce a canonical morphism $G\rightarrow G'$, where $G'$ contains no nontrivial unirational subgroup, and show that it induces a bijection ${H^1(K,G) \overset{\sim}{\to} H^1(K,G')}$. Specifically, our unirational variant implies the injectivity part via a standard twisting argument, whereas the surjectivity is obtained by extending \cite[Thm.~3, III.2.4]{CohoGalois} (originally stated for perfect fields) to the imperfect setting using \cite{conrad_gabber_prasad_2015, conrad_prasad_2016}.

\begin{thm}[Serre's Conjecture I for unirational groups]
\label{theorem main result}
Assume that $\kdim(K) \leq 1$. Then, for all unirational algebraic $K$-groups $G$, we have $H^1(K,G)=1$.
\end{thm}

\begin{proof}
Since $G$ is unirational, Proposition \ref{proposition h1 SolvableCase} implies that $H^1(K,G)$ is trivial as soon as $H^1(K,\mathcal{D}^\infty(G))$ is, so we may assume that $G$ is perfect. Moreover, by Theorem \ref{theorem unipotent radical unirational}, we know that ${R}_{u,K}(G)$ is unirational, so $H^1(K,{R}_{u,K}(G))=1$ by Proposition \ref{proposition h1 SolvableCase} again, as ${R}_{u,K}(G)$ is solvable. By dévissage, $H^1(K,G)$ is trivial if $H^1(K,G/{R}_{u,K}(G))$ is, so we can assume that ${R}_{u,K}(G)=1$ and that $G$ is pseudo-semisimple. Theorem \ref{theorem conjecture I pseudo semisimple} concludes.
\end{proof}

\subsection{Homogeneous spaces}
\label{subsection theorem serre}
Given a smooth algebraic group $G$ over $K$, a \emph{$G$-homogeneous space over $K$} is a nonempty smooth $K$-scheme of finite type endowed with an action of $G$ that is transitive on the $K_s$-points. 

\begin{thm}
\label{theorem surjective h1}
Assume that $\sdim(K) \leq 1$ and let $G$ be a smooth algebraic $K$-group. Then every $G$-homogeneous space $X$ is dominated by a $G$-torsor, that is, there exists a $G$-torsor $\widetilde{X}$ and a $G$-equivariant map $\widetilde{X} \to X$.
\end{thm}

The proof is identical to that of \cite[III.2.4, Thm.~3]{CohoGalois}, upon taking the largest smooth subgroup $(-)^\mathrm{sm}$ where appropriate and replacing Borel subgroups with pseudo-Borel subgroups (i.e. solvable pseudo-parabolic subgroups), see \S \ref{largestsmsubgrp}, \S \ref{pseudo-parabolics} and \S \ref{pseudo-borels}.

\begin{proof}
Fix a point $x_0 \in X(K_s)$ and consider the pairs $(H,a)$, called \emph{compatible pairs}, consisting of a smooth $K_s$-subgroup $H \subset G_{K_s}$ and a continuous map $a \colon \Gamma \to G(K_s)$ satisfying the properties: 
\begin{enumerate}
\item \label{1} for all $h \in H(K_s)$, $x_0 \cdot h = x_0$;
\item \label{2} for all $s \in \Gamma$, ${}^s x_0 = x_0 \cdot a_s$;
\item \label{3} for all $s$, $t \in \Gamma$, $a_s \, {}^s a_t \, a_{st}^{-1} \in H(K_s)$;
\item \label{4} for all $s \in \Gamma$, $a_s \, {}^s H \, a_s^{-1} = H$ as subgroups of $G_{K_s}$.
\end{enumerate}
Compatible pairs exist: take $H\coloneqq{H'}^\mathrm{sm}$, where $H'$ is the stabilizer of $x_0$ in $G_{K_s}$, and take for $a$ a continuous map such that ${}^s x_0 = x_0 \cdot a_s$ for all $s \in \Gamma$, which exists by homogeneity. 
\medskip

Let $(H,a)$ be a compatible pair with $H$ of minimal dimension. Let us show that $H=1$. This would imply that $a$ is a $1$-cocycle of $G$ defining a $G$-torsor that dominates $X$.

First of all, let us prove that the neutral component $H^0$ of $H$ is solvable. Let $L$ be the largest smooth connected linear algebraic $K_s$-subgroup of $H^0$. It is such that the quotient $H^0/L$ is a pseudo-abelian variety (see \cite[2.1.2.(4)]{KesBouSca}). Choose a pseudo-Borel subgroup $B$ of $L$. By definition, there exists a cocharacter $\lambda$ such that $B={R}_{u,K_s}(L)\, P_L(\lambda)$. For all $s \in \Gamma$, we have ${}^s B= {R}_{u,K_s}({}^s L)\, P_{{}^s L}({}^s \lambda)$, so $a_s \,{}^s B\, a_s^{-1}$ is also a pseudo-Borel of $a_s \,{}^sL \,a_s^{-1} = L$ (as subgroups of $H$). But by \cite[Thm.~C.2.5]{conrad_gabber_prasad_2015}, all pseudo-Borel subgroups of $L$ are conjugated by $L(K_s)$, so there exists $h_s \in L(K_s)$ such that $a_s \,{}^s B\, a_s^{-1} = h_s^{-1} \,B\, h_s$. We can choose the $h_s$ such that $h\colon\Gamma \to G(K_s)$, $s \mapsto h_s$ is continuous, so that we get a compatible pair $({N}_H(B)^\mathrm{sm},ha)$. Points $(1)$, $(2)$, and $(4)$ are clear enough. To check $(3)$, let $h_{s,t}' \coloneqq h_s \, a_s \, {}^s(h_t a_t) \, (h_{st} a_{st})^{-1}$
for all $s$, $t \in \Gamma$. The same calculations as in the proof of \cite[III.2.4, Lem.~3]{CohoGalois} show that $h_{s,t}' \in {N}_H(B)(K_s)$. Since ${N}_H(B)(K_s)={N}_H(B)^{\mathrm{sm}}(K_s)$, this shows $(3)$. Now, since $H$ is minimal, necessarily $H={N}_H(B)^\mathrm{sm}$, so $H={N}_H(B)$ and $L={N}_{L}(B)$. As $B$ is its own normalizer in $L$ by \cite[Prop.~3.5.7]{conrad_gabber_prasad_2015}, we get $L=B$. Hence, $H^0$ is an extension of a pseudo-abelian variety (which is commutative by \cite[Thm.~2.1]{Totaro}) by $B$, which is solvable. We conclude that $H^0$ is solvable.

Let us see next that $H$ itself is solvable. In view of the previous point, it is enough to show that $H/H^0$ is solvable. Let $S \subset (H/H^0)(K_s)$ be a $l$-Sylow subgroup seen as an algebraic subgroup of $H/H^0$. Write $B$ for the inverse image of $S$ in $H$. The group $Q=a_s \, {}^s B \, a_s^{-1}$ lies inside $H$ and its image $\overline{Q}$ in $H/H^0$ is another $l$-Sylow subgroup: by Sylow's theorems there exists $h_s \in H(K_s)$ such that $\overline{Q} = \overline{h_s} \, S \, \overline{h_s}^{-1}$, thus $Q= h_s^{-1} \, B \, h_s$. We can choose $h\colon s\mapsto h_s$ so that it is continuous. By doing the same verifications as above, we see that $({N}_H(B)^\mathrm{sm},ha)$ is a compatible pair. Hence $H={N}_H(B)$ and $H/H^0 = {N}_{H/H^0}(S)$. Every Sylow subgroup of $H/H^0$ being normal, we get that $H/H^0$ is a product of finite $l$-groups for several primes $l$, which are all solvable.

We finally show that $H$ is actually perfect. For every $s \in \Gamma$, the $s$-semilinear automorphism (see \cite[\S (1.2)]{FSS}) ${}^s G \to G$, $g \mapsto a_s \, {}^s g \, a_s^{-1}$ of $G$ restricts to a scheme isomorphism $\tau_s \colon {}^s H \to H$. The inverse $\rho_s$ of $\tau_s$ passes to the quotient by $\mathcal{D}(H)$ to yield an $s$-semilinear automorphism $\overline{\rho_s}$ of $H^\mathrm{ab}$. The collection of the $\overline{\rho_s}$'s defines a map $\overline{\rho} \colon \Gamma \to \mathrm{SAut}(H^\mathrm{ab}/K)$, where $\mathrm{SAut}(H^\mathrm{ab}/K)$ is the group of semilinear automorphism of $H^\mathrm{ab}$. From the definition of compatible pairs, $\overline{\rho}$ is also continuous (in the sense of \cite[Def.~(1.10)]{FSS}) and is a group homomorphism. Thus $\overline{\rho}$ defines a $K$-group structure on $H^\mathrm{ab}$ by \cite[Rmk.~(1.15)]{FSS}. Denote by $C$ the $K$-descent of $H^\mathrm{ab}$ defined by $\overline{\rho}$. For all $s$, $t \in \Gamma$, write $h_{s,t} = a_s {}^s a_t a_{st}^{-1} \in H(K_s)$, 
and $\overline{h_{s,t}}$ for its image in $H^\mathrm{ab}(K_s)=C(K_s)$.
The exact same calculations as in the proof of \cite[III.2.4, Lem.~5]{CohoGalois} imply that $(\overline{h_{s,t}})_{s,t}$ is a $2$-cocycle of $C$. However, $H^2(K,C)=1$ because of \cite[Cor.~A.4]{Cesnavicius}. Again, as in the proof of \cite[III.2.4, Lem.~5]{CohoGalois}, we build $a' \colon \Gamma \to G(K_s)$ such that $(\mathcal{D}(H),a')$ is a compatible pair, proving that $H=\mathcal{D}(H)$.

We conclude that $H$ is solvable and perfect, so $H=1$ and $a$ is a $1$-cocycle of $G$.
\end{proof}

\begin{coro}
\label{theorem surjective h1 strongly wound}
Assume that $\sdim(K) \leq 1$ and let $f\colon G\rightarrow G'$ be a surjective map of smooth algebraic $K$-groups. Then $f$ induces a surjection \[ H^1(K,G) \to H^1(K,G'). \]
\end{coro}

\begin{proof}
 Every $G'$-torsor can be seen as a $G$-homogeneous space, thus is dominated by a $G$-torsor by Theorem \ref{theorem surjective h1}.
\end{proof}

\subsection{Strongly wound quotient}
\label{subsection strongly wound quotient}
Let $G$ be a smooth algebraic $K$-group. There is a largest unirational $K$-subgroup $G^{\mathrm{uni}}$ of $G$ by \cite[2.1.2.(9)]{KesBouSca}. It is affine, smooth, connected, and normal in $G$. If $[K:K^p] \leq p$, then the quotient $G/G^{\mathrm{uni}}$ has no nontrivial unirational subgroup. Indeed, if $H \subset G/G^{\mathrm{uni}}$ is unirational, then the inverse image of $H$ by the quotient homomorphism $G \to G/G^{\mathrm{uni}}$ is unirational by \cite[Thm.~2.4]{Rosengarten_extuni}, thus $H=1$.

Suppose now that $G$ is also affine and connected. If $G^\mathrm{uni}=1$, we say that $G$ is \emph{strongly wound}. In this case, it is necessarily wound unipotent because it cannot contain a nontrivial torus nor a split unipotent group. If $[K:K^p] \leq p$, the group $G/G^{\textrm{uni}}$ is strongly wound, and the map $G \to G/G^{\textrm{uni}}$ is then called the \emph{strongly wound quotient} of $G$.

\begin{thm}
\label{corollary bijective h1 strongly wound}
Assume that $\kdim(K) \leq 1$ and let $G$ be a smooth algebraic $K$-group. Then the quotient homomorphism $G \to G/G^{\mathrm{uni}}$ induces a bijection $$H^1(K,G) \overset{\sim}{\rightarrow} H^1(K,G/G^{\mathrm{uni}}).$$ 
\end{thm}

\begin{proof}
By Corollary \ref{theorem surjective h1 strongly wound}, the map $H^1(K,G) \to H^1(K,G/G^{\mathrm{uni}})$ is surjective. Furthermore, every twist of $G^{\mathrm{uni}}$ is \mbox{$K_s$-unirational}, hence it is also {$K$-unirational} as recalled in $\S$\ref{paragraphUnirat}. Thus its $H^1$ is trivial by Theorem \ref{theorem main result}. The classical twisting argument proves the desired injectivity. 
\end{proof}

\begin{coro}
In the setting of Theorem \ref{corollary bijective h1 strongly wound}, suppose furthermore that the identity component $G^0$ is unirational. then the quotient homomorphism $G \to G/G^0$ induces a bijection
$$H^1(K,G) \to H^1(K,G/G^0).$$
\end{coro}

\bibliographystyle{alpha}
\bibliography{bibliography.bib}

\end{document}